\documentclass{amsart}
\usepackage{hyperref}
\usepackage{rotating}
\usepackage[all,cmtip]{xy}

\theoremstyle{plain} 
\newtheorem{theorem}{Theorem}[section]
\newtheorem{proposition}[theorem]{Proposition}
\newtheorem{corollary}[theorem]{Corollary}
\newtheorem{lemma}[theorem]{Lemma}

\theoremstyle{definition} 

\newtheorem{remark}[theorem]{Remark}
\newtheorem{definition}[theorem]{Definition}

\newcommand{\C}{\mathbb{C}}

\DeclareMathOperator{\Sym}{Sym}

\def\cC{\mathbb{C}}

\def\cL{\mathbb{L}}

\def\GL{\operatorname{GL}}

\begin{document}

\title{Motivic classes of generalized Kummer schemes via relative power structures}
\author{Andrew Morrison and Junliang Shen}

\address{Departement Mathematik, ETH Z\"urich}
\email{andrewmo@math.ethz.ch}

\address{Departement Mathematik, ETH Z\"urich}
\email{junliang.shen@math.ethz.ch}

 \begin{abstract} 
 We develop a power structure over the Grothendieck ring of varieties relative to an abelian monoid, which allows us to compute the motivic class of the generalized Kummer scheme. We obtain a generalized version of Cheah's formula for the Hilbert scheme of points, which specializes to Gulbrandsen's conjecture for Euler characteristics. Moreover, in the surface case we prove a conjecture of G\"{o}ttsche for geometrically ruled surfaces, and we obtain an explicit formula for the virtual motive of the generalized Kummer scheme in dimension three.
 \end{abstract}

\maketitle

\thispagestyle{empty}

\section{Introduction.}

We work over the complex numbers $\cC$. Let $X\to A$ be a Zariski locally trivial fibration with base $A$ a $g$-dimensional complex abelian variety and fiber $Y$ an $r$-dimensional smooth quasi-projective variety. Then the Hilbert scheme of $n$ points on $X$ is an iso-trivial fibration over $A$. The map $\pi_n : \textrm{Hilb}^n(X) \to A$ equals the Hilbert--Chow map $\textrm{Hilb}^n(X) \to \Sym^n(X) $, followed by projection $\Sym^n(X)\to \Sym^n(A)$, and addition $\Sym^n(A)\to A$. Following \cite{B} we call the fiber over zero the \textit{generalized Kummer scheme} of $n$ points 
\[K_n(X) :=\pi_n^{-1}(0_A) \subset \textrm{Hilb}^n(X).\]

When $g = 0$, the generalized Kummer scheme $K_n(X)$ is exactly the Hilbert scheme of points $\textup{Hilb}^n(X)$. Its cohomology was studied by Cheah \cite{Ch}. When $A$ is an abelian surface and $X = A$, the cohomology of $K_n(A)$ was first studied by G\"ottsche \cite{Lecture} using the Weil conjecture. Then the Hodge numbers were computed by G\"ottsche and Soergel \cite{GS} using the decomposition theorem for perverse sheaves. Both methods rely on the smoothness of $K_n(A)$. The Euler characteristics of $K_n(A)$ were also computed by Debarre \cite{De} and Gulbrandsen \cite{Gu} in different ways. 

Recently \cite{S} the second author provides a general formula for the Euler characteristics of the schemes $K_n(X)$, proving a conjecture of Gulbrandsen for abelian varieties \cite{Gul}. In particular, when $X$ is a surface or 3-fold and $g\geq 1$ we have
\[ \chi(K_n(X)) = \chi(Y)\cdot n^{2g-1}\cdot \sum_{d|n} d^{r+g-1}, \]
which is related to Donaldson--Thomas invariants of abelian 3-folds by \cite{Gul}. 

In this paper, we give a refinement of the formula for Euler characteristics $\chi([K_n(X)])$ in arbitrary dimension. Theorem \ref{thmgeneral} expresses the class $[K_n(X)]$ in terms of  the data of punctual Hilbert schemes in the Grothendieck ring $K_0(\textup{HS})$ of Hodge structures. This generalizes Cheah's formula and gives a new proof of Gulbrandsen's conjecture.  

When $X$ is a surface or 3-fold the following equality in $K_0(\textup{HS})[\mathbb{L}^{-\frac{1}{2}}]$ is given by Theorem \ref{surfkummer} and Theorem \ref{thm3fold},
\[ [A]_{\textrm{vir}}\cdot [K_n(X)]_{\textrm{vir}} = \sum_{\alpha \vdash n } g(\alpha)^{2g} \prod_{i\geq 1} \Sym^{b_i(\alpha)}([X\times \mathbb{P}^{(i-1)(r+g-2)}]_{\textrm{vir}})\]
here $b_i(\alpha)$ is the number of parts of $\alpha$ of size $i$, $g(\alpha) = \textrm{gcd}\{i : b_i(\alpha)\neq 0\}$, and for smooth varieties as above the virtual class $[M]_\textrm{vir}$ equals $\mathbb{L}^{-\frac{\dim(M)}{2}}\cdot[M]$. 

We can now take a refinement of the Euler characteristic such as the Hodge--Deligne polynomial. In the surface case $K_n(X)$ is smooth and when $g=0$ we recover G\"{o}ttsche's formula \cite[Theorem 1.1]{G}, when $g=2$ we recover G\"{o}ttsche and Soregel's result \cite[Theorem 7]{GS}, and when $g=1$ we have a proof of a conjecture of G\"{o}ttsche \cite[Conjecture 2.4.17]{Lecture}. In the 3-fold case $K_n(X)$ is a locally complete intersection, we define its virtual motive in Section 4 following \cite{BBS}. In particular when $g=0$ we recover the result of Behrend, Bryan, and Szendr\H{o}i \cite[Theorem 3.3]{BBS} while for $g\geq 1$ the formulas are new. For example, taking $X=K3\times E$ to be the product of a $K3$ surface and an elliptic curve and $X \to E$ to be the projection, we can compute the $\chi_{y}$ genus
\[\chi_{-y} ([K_n(X)]_{\textrm{vir}}) = (y^{-\frac{n}{2}}+y^{-\frac{n}{2}+1}+\cdots+y^{\frac{n}{2}}) \cdot \sum_{d\cdot m = n} d^2 (y^{-\frac{m}{2}} + 22 + y^{\frac{m}{2}})\]
here we see a $y\longleftrightarrow y^{-1}$ symmetry. Indeed for $X$ compact the Hodge numbers of $[K_n(X)]_{\textrm{vir}}$ have the same symmetries as that of a $3n-g$ dimensional compact K\"{a}hler manifold. In Section \ref{propsofhodge} other properties are also deduced.

In the paper our strategy is to work in the Grothendieck ring $K_0(\textup{Var}/A)$ of varieties relative over $A$. Here we develop the power structure and prove a few lemmas that combined are enough to deduce the main results of the paper. By computing the motivic class of the Hilbert scheme relative over $A$ and then pulling back over $0$ we deduce the above result for Kummer schemes via Lemma \ref{symkummer}. The final section of the paper includes an extension of the results to the motivic classes of stacks of torsion sheaves and we give explicit formulas for curves, surfaces, and 3-folds. 

\subsection{Acknowledgement}
We thank Georg Oberdieck, Rahul Pandharipande, and Qizheng Yin for discussions related to abelian geoemetries and motivic invariants.

This work was conducted in the research group of Pandharipande at the ETH Z\"{u}rich. A.M was supported by the Swiss National Science Foundation grant SNF 200021143274. J.S was supported by the grant ERC-2012-AdG-320368-MCSK.

\section{A Power Structure on Motives over an Abelian Monoid.}

\subsection{Relative Motives}
Let $(\mathfrak{M},+,0)$ be an abelian monoid. The Grothendieck ring $K_0(\textup{Var}/\mathfrak{M})$ has generators $[X\to \mathfrak{M}]$ and the usual cut and paste relations \cite{DM}. However the product of two classes $[X\to\mathfrak{M}]$ and $[Y\to\mathfrak{M}]$ equals
\[ \xymatrix{ X \times Y \ar[r] &  \mathfrak{M} \times \mathfrak{M} \ar[r]^{\hspace{0.4cm}+} & \mathfrak{M}}. \]
The effective classes $S_0(\textup{Var}/\mathfrak{M})$ form a sub-semi-ring of $K_0(\textup{Var}/\mathfrak{M})$. We describe a geometric power structure on this semi-ring. In general given a (semi)-ring $R$ a power structure is defined to be a map
\[\left(1+ tR[|t|] \right)\times R \to \left(1+ tR[|t|] \right), \]
denoted by $(A(t),r)\mapsto A(t)^r$ satisfying five axioms
\begin{enumerate} 
 \item $A(t)^0 = 1$,
 \item $A(t)^1 = A(t)$,
 \item $A(t)^n\cdot B(t)^n = \left( A(t)\cdot B(t)\right)^{n}$,
 \item $A(t)^{n+m} = A(t)^{n}\cdot A(t)^{m}$,
 \item $A(t)^{nm} = \left(A(t)^n\right)^m$.
\end{enumerate}  
To define the power structure on $S_0(\textup{Var}/\mathfrak{M})$, let $A(t) = \sum_{i\geq 0} [A_i]t^i$ be a series and $[M]$ an exponent such that the $A_i$ and $M$ are varieties over $\mathfrak{M}$. Then the series $A(t)^{[M]}$ has $t^n$ coefficient given by summing classes $B_\alpha$ over all the partitions $\alpha \vdash n$ 
\[ B_\alpha = \left( \prod_{i\geq 1} M^{b_i(\alpha)} \setminus \Delta \right) \times_{S_\alpha}  \left( \prod_{i\geq 1} A_i^{b_i(\alpha)} \right) \]
here $b_i(\alpha)$ is the number of parts of $\alpha$ of size $i$, $\Delta$ is the large diagonal, and the symmetric group $S_\alpha$ acts on the factors in the obvious way. To define the $\mathfrak{M}$ structure on the variety $B_\alpha$ we know that each of the factors has the structure of an $\mathfrak{M}$ variety. We take the projectors $p_{M,i}$ (and $p_{A,i}$) from the $i$th factor on the left (reps. right) side composed with the addition map to $\mathfrak{M}$. The map $p_\alpha : B_\alpha \to \mathfrak{M}$ is given by 
\[ p_\alpha = \sum_{i}\left( i \cdot p_{M,i} + p_{A,i}\right). \]
Now $A(t)^N$ has a geometric description in terms of collections of charged particles $(S,\phi)$ on $M$. Here $S$ is a finite subset of $M$ and
\[ \phi : S \to \coprod_{i=1}^\infty A_i .\]
We consider $\phi(s)$ as the internal state of particle $s\in S$. Every such particle $s\in S$ with $\phi(s)\in A_i$ has a multiplicity and weight
\begin{eqnarray*}
m(s,\phi) & = & i \in \mathbb{N} \\
w(s,\phi) & = & i\cdot p_{M\to \mathfrak{M}}(s) + p_{A_i\to \mathfrak{M}}(\phi(s)) \in \mathfrak{M}.
\end{eqnarray*}
The $t^n$ coefficient of $A(t)^N$ parameterizes collections $(S,\phi)$ with total multiplicity $n$ and the map to $\mathfrak{M}$ is given by the total weight.  

\begin{theorem} The above geometric construction is a power structure on $S_0(\textup{Var}/\mathfrak{M})$.
\end{theorem}
\begin{proof} The basic version of this theorem was the main result of \cite{GZLMH} where this result was proven for the abelian monoid $\mathfrak{M} = \textup{Spec}(\cC)$ a point. 

It is routine to check that each of the five bijections in \cite[Theorem 1]{GZLMH} preserve the total weight $w(S,\phi)$ of the configurations $(S,\phi)$ so axioms $(1)-(5)$ hold for the power structure defined on the ring $S_0(\textup{Var}/\mathfrak{M})$.
\end{proof}

\begin{theorem}\label{extthm} The power structure on $S_0(\textup{Var}/\mathfrak{M})$ uniquely extends to $K_0(\textup{Var}/\mathfrak{M})$.
\end{theorem}
\begin{proof} For each series $A(t)$ with coefficients in $K_0(\textup{Var}/\mathfrak{M})$ there exist series $B(t)$, $C(t)$ with coefficients in $S_0(\textup{Var}/\mathfrak{M})$ such that $A(t)\cdot B(t) = C(t)$. It suffices to define $B(t)^{-[M]}$ for $M$ a variety. For series with leading coefficient $1$ this is uniquely determined by 
\[ B(t)^{-[M]} := 1/B(t)^{[M]}. \]
\end{proof}

Next we prove two lemmas useful for formal manipulations.

\begin{lemma}\label{lef} Define the Lefschetz motive $\mathbb{L} = \left[ \mathbb{A}^1 \to 0\in \mathfrak{M}  \right] $ and $s\in \mathbb{N}$ then
\[ A(\mathbb{L}^st)^{[M]} =  \left. \left(A(t)^{[M]} \right) \right|_{t\mapsto \mathbb{L}^st} .\]
\end{lemma}
\begin{proof} The coefficients on the left hand side are of the form
\[ \left( \prod_{i\geq 1} M^{b_i(\alpha)} \setminus \Delta \right) \times_{S_\alpha}  \left( \prod_{i\geq 1} \left(A_i \times \mathbb{A}^i\right)^{b_i(\alpha)} \right) \]
the coefficients on the right hand side are of the form
\[ \mathbb{A}^n\times \left( \prod_{i\geq 1} M^{b_i(\alpha)} \setminus \Delta \right) \times_{S_\alpha}  \left( \prod_{i\geq 1} A_i ^{b_i(\alpha)} \right) . \]
We have two vector bundles of rank $n$ over the same base. By Hilbert's Theorem 90 they split in the Zariski topology and so have the same motivic class. 
\end{proof}

\begin{lemma}\label{npow} For $n\in \mathbb{N}$ define the map $n: \mathfrak{M} \to \mathfrak{M} $ as multiplication by $n$. This induces a push-forward $n_*: K_0(\textup{Var}/\mathfrak{M}) \to  K_0(\textup{Var}/\mathfrak{M})$ so that \[ A(t^n)^{[M]} = \left. A(t)^{n_*[M]} \right|_{ t \mapsto t^n }. \]
\end{lemma}
\begin{proof} When $n=1$ the result is immediate. In general for $n>1$ we can prove the results for effective classes by the uniqueness of the extension constructed in Theorem \ref{extthm}. The case of effective classes follows from the geometric description above by comparing the weight of the charged particles on either side.
\end{proof}

\begin{definition} We define the zeta function of $[X\to\mathfrak{M}] \in K_0(\textup{Var}/\mathfrak{M})$ by
\begin{equation*}
\begin{split}
 \zeta_{[X\to\mathfrak{M}]}(t) := & \left( 1 + t + t^2 + t^3\cdots \right)^{[X\to\mathfrak{M}]} \\
                                               = &\sum_{n \geq 0} [\textup{Sym}^n(X) \xrightarrow{+} \mathfrak{M}]t^n.
 \end{split}
\end{equation*} 
This series has some nice properties, for example we have $\zeta_{[A]}(t)^{[B]} = \zeta_{[B]\cdot[A]}(t)$ and each series $A(t)$ has a unique factorization $A(t) = \prod_{n\geq 1} \zeta_{[A_n]}(t^n)$. 
\end{definition}

\begin{lemma}[cf. \cite{G}]\label{zeta}  For the above power structure on $ K_0(\textup{Var}/\mathfrak{M})$ we have
\[ \zeta_{[A]}(t)^{\mathbb{L}^k} = \zeta_{\mathbb{L}^k[A]}(t) = \zeta_{[A]}(\mathbb{L}^kt) \textup{ for all } k\in \mathbb{N}.\]
\end{lemma}
\begin{proof} We must show that $[\textup{Sym}^n(\mathbb{A}^k\times X)] = \mathbb{L}^{n\cdot k }\cdot [\textup{Sym}^n(X)]$. Stratifying the left hand side into a collection of vector bundles this follows simply from Hilbert's Theorem 90 as in Lemma \ref{lef}. 
\end{proof}

\begin{theorem}\label{localps} There exists a unique extension of the power structure to the ring $K_0(\textup{Var}/\mathfrak{M})[\mathbb{L}^{-1}]$ with the property that 
\[ \zeta_{[A]}(t)^{\mathbb{L}^k} = \zeta_{\mathbb{L}^k[A]}(t) = \zeta_{[A]}(\mathbb{L}^kt) \textup{ for all } k\in \mathbb{Z}.\]
\end{theorem}
\begin{proof}Given a series $A(t)$ with coefficients in $K_0(\textup{Var}/\mathfrak{M})[\mathbb{L}^{-1}]$ there is a factorization $A(t) = \prod_{n\geq 1} \zeta_{\mathbb{L}^{-a_n}[A_n]}(t^n)$ with $[A_n]\in K_0(\textup{Var}/\mathfrak{M})$. Taking an exponent $[M]\in K_0(\textup{Var}/\mathfrak{M})$ we see that
\begin{eqnarray*}
 \zeta_{\mathbb{L}^{-a_n}[A_n]}(t^n)^{\mathbb{L}^{-m}[M]} &=& \left.  \zeta_{\mathbb{L}^{-a_n}[A_n]}(t)^{n_*(\mathbb{L}^{-m}[M])} \right|_{t \mapsto t^n} \\
       &=& \left.  \zeta_{(\mathbb{L}^{-m-a_n}) \cdot[A_n]}(t)^{n_*([M])} \right|_{t \mapsto t^n} \\
      &=& \left.  \zeta_{n_*([M])\cdot[A_n]}(\mathbb{L}^{-m-a_n}  t) \right|_{t \mapsto t^n} 
\end{eqnarray*}
this gives a unique extension of the power structure. 
\end{proof}

Following Behrend, Bryan, and Szendr\H{o}i \cite{BBS} we can further extend the power structure to the rings $K_0(\textup{Var}/\mathfrak{M})[\mathbb{L}^{-1/2}]$ where the elements $\mathbb{L}^{-1/2}$ are adjoined formally. Indeed the above proof goes through as in \cite[Section 1.5]{BBS} and there is a unique extension of the power structure such that
\[  \zeta_{[A]}(t)^{\mathbb{L}^k} = \zeta_{\mathbb{L}^k[A]}(t) = \zeta_{[A]}(\mathbb{L}^kt) \textup{ for all } k\in (1/2)\mathbb{Z}. \]

\subsection{Kummer fibers.} Given a motivic class $[X\to \mathfrak{M}]\in K_0(\textup{Var}/\mathfrak{M})[\mathbb{L}^{-1/2}]$ we are interested in the fiber over zero. Let $0 : \textup{Spec}(\cC) \to \mathfrak{M}$ be the map sending a point to the additive identity in $\mathfrak{M}$. This induces the pullback
\[0^*:  K_0(\textup{Var} /\mathfrak{M})[\mathbb{L}^{-1/2}] \to K_0(\textup{Var}_\cC )[\mathbb{L}^{-1/2}]\]
sending $ [X\to \mathfrak{M}] $ to $[\textup{Spec}(\cC)\times_{\mathfrak{M}} X \to \textup{Spec}(\cC) ]$ when $X$ is a variety. Note that the map $0^\ast$ is a group homomorphism.

We have seen the appearance of symmetric products in the definition of the power structure. Here we prove a lemma to be used later regarding symmetric products of abelian varieties and their ``Kummer fibers'', i.e. fibers over zero. While we would like to prove these results in the Grothendieck ring of motivic classes we can only find a proof at the level of Hodge structures.

\begin{lemma}[cf. \cite{GS} Theorem 7, Step (7)] \label{symkummer} 
Let $X_m\to A$ for $m=1,\ldots, l$ be a collection of Zariski trivial fiber bundles over a $g$-dimensional complex abelian variety $A$ then,
\[ [A]\cdot0^*\left( \prod_{m=1}^l m_*\left([\textup{Sym}^{r_m}(X_m)\to A] \right)\right) = g(r)^{2g} \cdot  \prod_{m=1}^l [\textup{Sym}^{r_m}(X_m)]  \]
in the ring $K_0(\textup{HS})$ of Hodge structures with $g(r) = \textup{gcd}(\{ m: r_m \neq 0 \}).$
\end{lemma}
\begin{proof} It suffices to prove the case $g\geq 1$. The proof follows step $7$ in G\"{o}ttsche and Soergel's proof of Theorem 7 in \cite{GS}. Without loss of generality we may assume that $X_m = Y_m\times A$ as all the above classes are independent of the Zariski fibration.

Assume that $g(r)=1$ and let $X^{(r)} = \prod_{m=1}^l\textup{Sym}^{r_m}(X_m) $ with the map to $A$ given by projection to $A^{(r)}$ followed by the addition map scaling the $m$-th factor by $m$. Define $K^{(r)}$ to be the fiber of $X^{(r)}$ over zero. The natural quotient map $A\times K^{(r)} \to X^{(r)}$ has fiber given by the $n = \sum_{m}mr_m$ torsion points $A(n)$. The pre-quotient represents the Hodge structure on the left hand side above while the base represents the right hand side. It is enough to show that $A(n)$ acts trivially on the cohomology with compact support of $A\times K^{(r)}$. In fact since the quotient map $A\times K^{(r)} \to X^{(r)}$ is the equivariant  $S_{(r)}$ quotient of the map of cartesian products $A\times K^{r} \to X^{r}$, it is enough to show that $A(n)$ acts trivially on the cohomology with compact support of the space induced form the cartesian product $A\times K^{r}$. The factors $Y_m$ can now be removed and the proof is identical to \cite{GS}. 

Likewise in the $g(r)>1$ case we reduce to the $g(r)=1$ case. When $g(r)>1$ the pull back over zero is equal to 
\[  \sum_{x\in A(g(r))} x^* \left( \prod_{m=1}^l \left(\frac{m}{g(r)}\right)_*\left([\textup{Sym}^{r_m}(X_m)\to A] \right)\right) \]
where we reduce the multiplicity of all push forwards by a factor of $g(r)$ but allow the image to lie in the $g(r)$-torsion points of $A$ of which there are $g(r)^{2g}$. Each of these new terms now satisfies $g(r)=1$ and we have the result. 
\end{proof}

For later use we determine Hodge--Deligne polynomials of symmetric products. Recall that the Hodge--Deligne polynomial is a ring homomorphism
\[
E: K_0(\textup{Var}) \rightarrow \mathbb{Z}[x,y]
\]
defined on generators by 
\[
E \left( [X]; x,y \right) = \sum_{p,q}x^py^q \sum_{i} (-1)^ih^{p,q}(H_c^i(X, \mathbb{Q})).
\]
The coefficient of $x^py^q$ is denoted by $e^{p,q}$. If $X$ is a smooth projective variety, then clearly $e^{p.q}([X]) = (-1)^{p+q}h^{p.q}(X)$. For any motivic class $[X]\in K_0(\textup{Var}_\cC)$, we call $h^{p,q}\left([X]\right): = (-1)^{p+q}e^{p,q}\left([X]\right)$ the Hodge numbers of $[X]$. All Hodge numbers of symmetric products of smooth quasi-projective varieties can be computed via Macdonald's formula,
\begin{equation*}
\sum_{n=0}^{\infty}E([\textup{Sym}^n(X)];x,y)t^n = \prod_{p,q}\Big{(} \frac{1}{1-x^py^qt}\Big{)}^{e^{p,q}([X])},
\end{equation*}
equivalently,
\begin{equation}\label{Macdonald}
E\left( \left[\textup{Sym}^n(X); x,y\right] \right) = \sum_{\alpha \vdash n} \prod_{j} \frac{1}{j^{b_j(\alpha)}\cdot b_j(\alpha)!} E\left([X]; x^j, y^j  \right)^{b_j(\alpha)}
\end{equation}
where $b_{m}(\alpha)$ stands for the number of parts in the partition $\alpha$ of size $m$, see \cite[Proposition 1.1]{Ch}. Hence we obtain the following lemma for the $\chi_y$ genus.

\begin{lemma}\label{Kummerchiy}
Let $X\to A$ be a Zariski locally trivial fibration over a $g$-dimensional complex abelian variety with fiber $Y$, then we have
\[
\chi_{-y}\left( \frac{[\textup{Sym}^n(X)]}{[A]}\right) := \left(\frac{E([\textup{Sym}^n(X)]; x,y)}{E([A] ; x,y)}\right) \Big{|}_{x=1} = \chi_{-y^n}(Y) \cdot n^{g-1} \left( \sum_{i=0}^{n-1}y^i  \right)^g.\]
\end{lemma}
\begin{proof}
By the formula (\ref{Macdonald}), only the partition of the form $\alpha = (n^1)$ contributes to the left-hand side. Thus we get
\begin{equation*}
\begin{split}
\chi_{-y}\left( \frac{[\textup{Sym}^n(X)]}{[A]}\right)   = & \left( \frac{1}{n} \cdot \frac{E\left([X]; x^n, y^n  \right)}{E([A]; x,y)}\right)\Big{|}_{x=1} \\
  = & \left( \frac{E([Y] ; 1,y^n)}{n}\right) \cdot \left( \frac{E\left([A]; x^n, y^n  \right)}{E([A]; x,y)}\right)\Big{|}_{x=1}\\
  = & \chi_{-y^n}(Y) \cdot n^{g-1} (1+y+ \cdots +y^{n-1})^g. \qedhere
  \end{split}
 \end{equation*}
\end{proof}

\subsection{Stacks.} This subsection is only needed in the final application of the paper and may be skipped. 

We develop the power structure for stacks over $\mathfrak{M}$ with affine stabilizers. The ring $K_0(\textup{St}/\mathfrak{M})$ is generated by all such stacks $[X\to \mathfrak{M}]$, with the cut and paste relations, together with the multiplication defined above. In addition the ring $K_0(\textup{St}/\mathfrak{M})$ has the relation 
\[ [E\to \mathfrak{M}]  =  \mathbb{L}^r \cdot [X\to \mathfrak{M}]  \]
when $\pi : E\to X$ is a vector bundle of rank $r$, and $\mathbb{L} = [\mathbb{A}^1\to 0 \in \mathfrak{M}]$ as usual.

\begin{theorem}[cf. \cite{E}]\label{stack} The ring $K_0(\textup{St}/\mathfrak{M})$ is isomorphic to 
\[  K_0(\textup{Var}/\mathfrak{M})[[\GL_n(\cC)\to 0\in \mathfrak{M}]^{-1} : n\geq 1]. \]
\end{theorem}
\begin{proof} For any global quotient stack $[[X/\GL_n(\cC)] \to \mathfrak{M} ]$ we have the relation 
\[[X \to \mathfrak{M} ] = [\GL_n(\cC) \to 0\in \mathfrak{M}] \cdot [[X/\GL_n(\cC)] \to \mathfrak{M} ].\] 
This is quickly proven using the new relation for vector bundles. If a stack $\mathcal{X}$ over $\mathfrak{M}$ can be represented in two ways, say as $[X/\GL_n(\cC)]$ and $[X'/\GL_{n'}(\cC)]$, then by taking the fiber product we get 
\[  \xymatrix{ X'' \ar[r]  \ar[d]& X' \ar[d]  \\ X \ar[r] & \mathcal{X}  \\ } \]
so $X''$ is a $\GL_{n'}(\cC)$ torsor over $X$ and a $\GL_{n}(\cC)$ torsor over $X'$. We have
\[ [\GL_{n'}(\cC) \to 0\in \mathfrak{M}] \cdot [X \to \mathfrak{M}] =[\GL_{n}(\cC) \to 0\in \mathfrak{M}] \cdot [X' \to \mathfrak{M}] .\]
As shown in \cite{K} every stack with affine stabilizers can be stratified by global quotients of the form $[X/\GL_n(\cC)]$. So the above relations give a well defined isomorphism sending the stack $[X/\GL_n(\cC)] \to \mathfrak{M} ]$ to $[X \to \mathfrak{M} ]\cdot [\GL_n(\cC)\to 0\in \mathfrak{M}]^{-1}$.
\end{proof}

This shows that the ring $K_0(\textup{ST}/\mathfrak{M})$ can be written as the middle term in a sequence of localized rings
\[ K_0(\textup{Var}/\mathfrak{M}) \to K_0(\textup{Var}/\mathfrak{M})[\mathbb{L}^{-1},(1-\mathbb{L}^{-n})^{-1}]\to K_0(\textup{Var}/\mathfrak{M})[[\mathbb{L}^{-1}]] \]
here the final ring is the completion with respect to the dimension filtration \cite{BD}.

\begin{theorem}[cf. \cite{E1}]\label{stackps} The power structure on $K_0(\textup{Var}/\mathfrak{M})[\mathbb{L}^{-1}]$ extends uniquely to one over $K_0(\textup{St}/\mathfrak{M})$ and $K_0(\textup{Var}/\mathfrak{M})[[\mathbb{L}^{-1}]]$.
\end{theorem}
\begin{proof} Given an element $[A]/[\GL_n] \in K_0(\textup{St}/\mathfrak{M})$ we define its zeta function as
\begin{eqnarray*}
\zeta_{[A]/[\GL_n]}(t) &=& (1+t+t^2+\cdots)^{[A]\mathbb{L}^{-n^2}\prod_{i=0}^{n-1} (1-\mathbb{L}^{-i})^{-1} } \\
 & = & (1+t+t^2+\cdots)^{[A]\mathbb{L}^{-n^2}\sum_{k\geq 0} l_k\mathbb{L}^{-k} } \\
 & = & \prod_{k\geq 0}\left( \zeta_{[A]\mathbb{L}^{-n^2-k}}(t)\right)^{l_k}. \\
\end{eqnarray*}
In particular we see that $\zeta_{[A]/[\GL_n]}(t) = 1 + ([A]/[\GL_n]) t + O(t^2)$ from which it follows that any series $A(t)\in 1 + K_0(\textup{St} / \mathfrak{M})$ can be factored as
\[ A(t) = \prod_{k\geq 1} \zeta_{[A_k]}(t^k) \textup{ where } [A_k] \in K_0(\textup{St}/ \mathfrak{M})\]
using the above definition of the zeta function it is also easy to check 
\[ \left( \zeta_{[A]}(t^n)\right)^{[B]} = \zeta_{n_*([B])[A]}(t^n) \textup{ for all } [A],[B]\in K_0(\textup{St}/\mathfrak{M}) \]
this determines the power structure uniquely. 
\end{proof}

The following lemma shows that the geometric essence of the power structure continues to hold in a limited sense.

\begin{lemma}[cf. \cite{BM} Lemma 6] \label{geomstack}
Let $M\to \mathfrak{M}$ be a variety and $A_n\to \mathfrak{M}$ be stacks and define $A(t) = \sum_{n\geq 0}[A_n\to \mathfrak{M}]t^n$. Then we have 
\[ A(t)^{[M\to \mathfrak{M}]} = \sum_{\alpha} [B_{\alpha} \to \mathfrak{M}]t^{|\alpha|} \]
where $|\alpha|$ is the size of partition $\alpha$ and $B_\alpha $ is defined geometrically as above.
\end{lemma}
\begin{proof} We need the result of Lemma \ref{lef} which carries over for stacks. Then the proof follows exactly the same lines as \cite{BM} using a double filtration on the ring of polynomials. 
\end{proof}





\section{Generalized Kummer schemes}
To begin consider a general abelian geometry $X\to A$ with $r+g = d$ and let $\textup{Hilb}^n(X) \to A$ be the Hilbert scheme of $n$ points on $X$. We arrange their motivic classes into a generating series $Z_X(t) = \sum_{n \geq 0} [\textup{Hilb}^n(X) \to A]t^n$. Then by the geometric definition of the power structure (cf. \cite{GZLMH2}) we have
\[  Z_X(t) =  \left( \sum_{n\geq 0} \left[\textup{Hilb}^n\left(\mathbb{A}_{0}^{d}\right)\to 0\right] t^n \right)^{[X\to A]}\]
and we hereby denote as $H_d(t)$ the sum inside the above bracket containing the punctual Hilbert schemes of points on $\mathbb{A}^d$.

\subsection{Surfaces.} Now we consider a surface $X\to A$ so that $r+g=2$. In this case Ellingsrud and Str{\o}mme \cite{ES} give a cell decomposition of the punctual Hilbert schemes so that 
\[ H_2(t) =  \prod_{k \geq 1}(1- \cL^{k-1}t^k)^{-1}. \]
Expanding this series gives a relative version of G\"{o}ttsche's result \cite[Theorem 1.1]{G}.
\begin{proposition}\label{relG} Let $X\to A$ be a surface then in $K_0(\textup{Var}/ A)$ we have
\[ [\textup{Hilb}^n(X)\to A] = \mathbb{L}^{n} \sum_{\alpha \vdash n } \prod_{m\geq 1} \mathbb{L}^{-b_{m}(\alpha)} m_* \left( [\textup{Sym}^{b_m(\alpha)}(X)\to A] \right) \]
where $b_{m}(\alpha)$ is the number of parts in $\alpha$ of size $m$.
\end{proposition}
\begin{proof} Using Lemma \ref{lef} and Lemma \ref{npow} this follows immediately on expanding the series for $H_2(t)^{[X\to A]}$.\end{proof}
Pulling back over the origin then gives a generalized version of \cite[Theorem 6]{GS} as follows, which proves G\"ottsche's conjecture for geometrically ruled surfaces. See  \cite[Conjecture 2.4.17]{Lecture}.

\begin{theorem}\label{surfkummer} Let $X\to A$ be a surface over a $g$-dimensional complex abelian variety and $g(\alpha) = \textup{gcd}(\{ i : b_i(\alpha) \neq 0 \})$. Then in $K_0(\textup{HS})$ we have
\[ [A]\cdot [K_n(X)] = \mathbb{L}^{n} \cdot \sum_{\alpha \vdash n } g(\alpha)^{2g}  \prod_{m\geq 1} \mathbb{L}^{-b_{m}(\alpha)} \left( [\textup{Sym}^{b_m(\alpha)}(X)] \right) . \]
\end{theorem}
\begin{proof} Follows immediately from Proposition \ref{relG} and Lemma \ref{symkummer}.\end{proof}

We compute the $\chi_y$ genus and the Euler characteristic by specializing the formula above as in \cite{Lecture}.
\begin{corollary} Let $X\to A$ be a surface over a $g$-dimensional abelian variety with fiber $Y$ and $g\geq 1$ then
\[\chi_{-y}(K_n(X))= y^n\cdot n^{g-1}\cdot \sum_{d\cdot m = n} y^{-m} \chi_{-y^{m}}(Y)\cdot d^{g+1} \left( \sum_{i=0}^{m-1} y^i \right) ^g.\]
In particular, the Euler characteristic
\[ \chi\left( K_n(X) \right) = \chi(Y) \cdot n^{2g-1}\cdot \sum_{d| n }   d. \]
\end{corollary}
\begin{proof} From Lemma \ref{Kummerchiy}, we see that only partitions of the form $n = (d^m)$ contribute in the specialization of Corollary \ref{surfkummer}, and in this case the contribution to the $\chi_y$ genus is 
\[d^{2g}y^{n-m}\cdot \chi_{-y} \left(  \frac{[\textup{Sym}^m(X)]}{[A]} \right) =  n^{g-1} d^{g+1}y^{n-m}\cdot \chi_{-y^m}(Y) \left( \sum_{i=0}^{m-1} y^i \right) ^g.\]
The Euler characteristic is obtained by setting $y=1$.
\end{proof}
From Theorem \ref{surfkummer}, we can show by the same argument as in \cite{Lecture} Corollaries 2.3.13 and 2.4.14 that the Hodge numbers $h^{p,q}([K_n(X)])$ become stable for $n > 2p$ or $n> 2q$. In particular, the Betti numbers $b_i([K_n(X)]) = \sum_{p+q=i} h^{p.q}([K_n(X)])$ become stable for $n>i$. A detailed proof will be given later for 3-folds in Section 4.4.
\subsection{General case}\label{gensection}
When $d>2$ the lack of a concrete description of $H_d(t)$ gives difficulties in computing Hodge numbers of generalized Kummer schemes. The special case appears in dimension 3 that we can study virtual Hodge numbers instead of normal ones, which will be illustrated in the next section. Here we formally write the generating series of punctual Hilbert schemes as
\[
H_d(t) = \prod_{k \geq 1} \left(  1- t^k    \right) ^{-[W_k]},
\]
where $[W_k] \in K_0(\textup{Var}_\C)$. Hence the Euler characteristics $w_k: = \chi([W_k])$ satisfy the equation
\begin{equation}\label{Euler}
\prod_{k \geq 1}(1-t^k)^{-w_k} = \sum_{n \geq 0} P_d(n)\cdot t^n,
\end{equation}
where $P_d(n)$ stands for the number of $d$-dimensional partitions of $n$, see \cite{Ch}.

We will express the Hodge numbers of $K_n(X)$ in terms of data of $[W_k]$. Here the $n$-th symmetric product of a class $[W]\in K_0(\textup{Var}_\cC)$ is defined to be
\[
\textup{Sym}^n([W]) : = \textup{Coeff}_{t^n} \left(  (1-t)^{-[W]}        \right).
\]

\begin{theorem}\label{thmgeneral}
Let $X \to A$ be a $d$-dimensional smooth quasi-projective variety over a $g$-dimensional abelian variety with smooth fiber $Y$. Then in $K_0(\textup{HS})$ we have 
\[
[A] \cdot [K_n(X)] = \sum_{\alpha \vdash n}g(\alpha)^{2g} \prod_{m \geq 1} [\textup{Sym}^{b_m(\alpha)}([X\times W_m]) ].
\]
\end{theorem}
\begin{proof}
The same proof as Theorem \ref{surfkummer} in the surface case.
\end{proof}
When $g=0$, the formula above is equivalent to the main theorem of \cite{Ch}.  When $g \geq 1$, this gives a proof of a formula conjectured by Gulbrandsen \cite{Gul} for Euler characteristics, which has been proven in \cite{S} in a different way.

\begin{corollary}[\cite{S}] \label{Gulconj}
When $g \geq 1$, we have
\[
\chi(K_n(X)) = \chi(Y)\cdot n^{2g-1} \cdot \sum_{d|n}dw_d,
\]
equivalently,
\[
\textup{exp}\left(    \sum_{n \geq 1} \frac{\chi(K_n(X))}{n^{2g}}  t^n  \right)
 = \left( \sum_{k\geq 0} P_d(k) \cdot t^k \right)^{\chi(Y)}.
\]
\end{corollary}
\begin{proof}
Only partitions of the form $\alpha = (d^m)$ contribute to the Euler characteristic by 
Lemma \ref{Kummerchiy}, and the contribution in this case is $\chi(Y)\cdot w_d\cdot d^{2g}\cdot m^{2g-1}$. This proves the first equation. For the second equation, we know from (\ref{Euler}) that
\begin{equation*}
\begin{split}
\chi(Y) \cdot \textup{log} \left(     \sum_{k\geq 0} P_d(k) \cdot t^k     \right) =&\chi(Y) \cdot \sum_{k \geq 1} w_k\left(   - \textup{log}(1-t^k)  \right)\\
=&   \sum_{n \geq 1}\left(   \sum_{n=m d}  \chi(Y) \cdot  \frac{w_d}{m}    \right)\cdot t^n \\
=&  \sum_{n\geq 1} \frac{ \chi(K_n(X))}{n^{2g}} \cdot t^n. \qedhere
\end{split}
\end{equation*}
\end{proof}

\section{Virtual motives in dimension three.}
In this section, we study an abelian geometry $X \rightarrow A$ with $r+g =3$. The main result of \cite{BBS} is that the virtual motive of the punctual Hilbert scheme $\textup{Hilb}^n\left(\mathbb{A}_{0}^{3}\right)$ can be expressed in terms of the Lefschetz motive $\cL$. By pushing forward the absolute virtual motive along the morphism $0: \textup{Spec}(\cC) \rightarrow A$, we get the virtual motive 
\[
\left[\textup{Hilb}^n\left(\mathbb{A}_{0}^{3}\right)\to 0 \right]_{\textup{vir}} \in K_0 \left(\textup{Var} /A \right) [\cL^{-\frac{1}{2}}]
\]
which by \cite{BBS} Theorem 3.7 and Proposition 4.2 satisfies the equation
\begin{equation*}
\sum_{n\geq 0} \left[\textup{Hilb}^n\left(\mathbb{A}_{0}^{3}\right)\to 0\right]_{\textup{vir}}  t^n
= \prod_{m=1}^{\infty} \prod_{k=0}^{m-1} \left(1- \cL^{-\frac{m}{2}+k-1}t^m\right)^{-1}.
\end{equation*}

\subsection{Relative virtual motives}
As in \cite{BBS} Definition 4.1, we use $\left[\textup{Hilb}^n\left(\mathbb{A}_{0}^{3}\right)\right]_{\textup{vir}}$ and the stratification of $\textup{Hilb}^n\left(X \right)$ to define the relative virtual motive 
\[
\left[\textup{Hilb}^n\left(X \right)\rightarrow A\right]_{\textup{vir}} \in K_0(\textup{Var}/A) [\cL^{-\frac{1}{2}}]
\]
so by the geometric definition of the power structure we have
\[\sum_{n\geq 0}\left[\textup{Hilb}^n\left(X \right)\rightarrow A\right]_{\textup{vir}}t^n =  \Big{(} \sum_{n\geq 0} \left[\textup{Hilb}^n\left(\mathbb{A}_{0}^{3}\right)\to 0 \right]_{\textup{vir}}  t^n \Big{)}^{[X \rightarrow A]}.\]

Expanding this series gives the relative virtual motives of the $3$-fold.
\begin{proposition} \label{3foldhilbert}
Let $X\to A$ be a $3$-fold then in $K_0(\textup{Var}/A)[\mathbb{L}^{-\frac{1}{2}}]$ we have
\[
\left[\textup{Hilb}^n\left(X \right)\rightarrow A\right]_{\textup{vir}}
=  \sum_{\alpha \vdash n }\prod_{m\geq 1} \mathbb{L}^{-\frac{b_m(\alpha)(m+2)}{2}} m_* \left( [\textup{Sym}^{b_m(\alpha)}(X\times \mathbb{P}^{m-1})\to A] \right).
\]
\end{proposition}

\begin{proof}
We have
\begin{equation*}
\begin{split}
\sum_{n\geq 0}\left[\textup{Hilb}^n\left(X \right)\rightarrow A\right]_{\textup{vir}}t^n
 = & \left(\prod_{m=1}^{\infty} \prod_{k=0}^{m-1} \left(1- \cL^{-\frac{m}{2}+k-1}t^m\right)^{-1}\right)^{[X\to A]}\\
 = & \prod_{m=1}^{\infty}  \left(\frac{1}{1- \cL^{-\frac{m}{2}-1}t^m}\right)^{[X\times \mathbb{P}^{m-1}\to A]}.
\end{split}
\end{equation*}
expanding the series using Lemma \ref{lef} and Lemma \ref{npow} we see the $t^n$ coefficient as claimed.
\end{proof}

\subsection{Motivic DT invariants}

Now we consider the virtual motive of the generalized Kummer scheme. The map $\pi_n: \textrm{Hilb}^n(X)\to A$ is an iso-trivial fibration. Over $D = \textrm{Spec}(\widehat{\mathcal{O}}_{A,0})$ we have a local neighborhood $ U = D \times_A \textrm{Hilb}^n(X) $ of $K_n(X)$ in $\textrm{Hilb}^n (X)$ isomorphic to $D\times K_n(X)$. The virtual motive of the smooth scheme $D$ equals $\mathbb{L}^{-\frac{g}{2}}$ and so we define the virtual motive of the generalized Kummer scheme $K_n(X)$ to be
\[
[K_n(X)]_{\textup{vir}} := \cL^{\frac{g}{2}} \cdot 0^* \left[\textup{Hilb}^n\left(X \right)\rightarrow A\right ] _{\textup{vir}} \in  K_0 \left(\textup{Var}_\C \right) [\cL^{-\frac{1}{2}}].
\]
We know from Proposition \ref{3foldhilbert} that
\begin{equation}\label{fiber3fold}
[K_n(X)]_{\textup{vir}} = \cL^{-\frac{n-g}{2}}\cdot\sum_{\alpha \vdash n}  \cL^{-l(\alpha)}\cdot0^*\left( \prod_{m\geq 1} m_*\left([\textup{Sym}^{b_m(\alpha)}(X\times \mathbb{P}^{m-1})\to A] \right)\right).
\end{equation}

Note that the Hodge--Deligne polynomial homomorphism 
\[
E: K_0(\textup{Var}_\cC) \rightarrow \mathbb{Z}[x,y]
\]
extends to a ring homomorphism
\[
E:  K_0(\textup{Var}_\cC)[\cL^{-\frac{1}{2}}] \rightarrow \mathbb{Z}[x,y,(xy)^{-\frac{1}{2}}].
\]
The following theorem determines the Hodge--Deligne polynomial of $[K_n(X)]_{\textup{vir}}$.
\begin{theorem} \label{thm3fold}
For a 3-fold $X \rightarrow A$ over a $g$-dimensional abelian variety, we have in $K_0(\textup{HS})[\cL^{-\frac{1}{2}}]$ that
\[
[A]_{\textup{vir}} \cdot [K_n(X)]_{\textup{vir}}  = \sum_{\alpha \vdash n} g(\alpha)^{2g}\prod_{m\geq 1} \mathbb{L}^{-\frac{b_m(\alpha)(m+2)}{2}} [\textup{Sym}^{b_m(\alpha)}{(X\times \mathbb{P}^{m-1})}].
\]
\end{theorem}
\begin{proof}
This is a consequence of (\ref{fiber3fold}) and Lemma \ref{symkummer}
\end{proof}

We say that $[K_n(X)]_{\textup{vir}}$ is a motivic Donaldson--Thomas invariant for the following reasons. When $g=0$, the class $[K_n(X)]_{\text{vir}} = [\textup{Hilb}^n(X)]_{\textup{vir}}$ is the motivic DT invariant in the sense of \cite{BBS}. In particular when $X$ is a Calabi--Yau 3-fold, the Euler characteristic of the class $[K_n(X)]_{\text{vir}}$ gives the ordinary degree 0 DT invariant. When $X =A$ is an abelian 3-fold, Gulbrandsen \cite{Gul} constructed a non-trivial virtual class on $K_n(A)$ whose degree is the twisted Euler characteristic $\chi^B(K_n(A))$. Our class $[K_n(A)]_{\textup{vir}}$ provides a refinement of Gulbrandsen's numerical DT invariant (see Lemma \ref{euler1}).

\subsection{Virtual Euler characteristic}
We study the virtual Euler characteristic of the Kummer fiber $K_n(X)$. For any class $[M] \in K_0(\textup{Var}_\cC)[\cL^{-\frac{1}{2}}]$, its $\chi_y$ genus is defined by
\[
\chi_{-y}([M]) = E\left([M]; x,y \right)\big{|}_{x^{\frac{1}{2}}=1} \in \mathbb{Z}[y^{\pm \frac{1}{2}}]
\]
and its Euler characteristic is 
\begin{equation*}
\begin{split}
\chi([M]) 
              =  \chi_{-y}([M])\big{|}_{y^{\frac{1}{2}}=-1} \in \mathbb{Z}.
\end{split}
\end{equation*}

\begin{lemma}\label{euler1}
We have
\[
\chi\left(\left[K_n(X)\right]_{\textup{vir}}\right) = \chi^B \left(K_n(X)\right) = (-1)^{n-g} \chi\left(K_n(X)\right).
\]
Here $\chi^B$ denotes the Euler characteristic weighted by the Behrend function \cite{Be}.
\end{lemma}
\begin{proof}
The proof of these statements for Hilbert schemes $\textup{Hilb}^n(X)$ can be adapted to $K_n(X)$. The first equality is obtained by \cite{BF} and the second equality is a consequence of \cite[Proposition 2.16]{BBS}.
\end{proof}

The following corollary computes the virtual $\chi_y$ genus of $K_n(X)$ as well as each Euler characteristic above using Theorem \ref{thm3fold}. 
\begin{corollary}\label{euler2}
When $g \geq 1$, we have 
\[
\chi_{-y}\left(\left[K_n(X)\right]_{\textup{vir}}\right) = n^{g-1}\sum_{n=md}d^{g+1}y^{-m-\frac{n-g}{2}}\cdot \chi_{-y^m}(Y) \left( \sum_{i=0}^{d-1} y^{im} \right) \left( \sum_{j=0}^{m-1} y^j    \right)^g.
\]
In particular, the virtual Euler characteristic 
\[
\chi\left(\left[K_n(X)\right]_{\textup{vir}}\right) = (-1)^{n-g}n^{2g-1}\cdot \chi(Y)\sum_{d|n}d^2.
\]
\end{corollary} 
\begin{proof}
Only partitions of the form $n= (d^m)$ contribute to the specialization of Theorem \ref{thm3fold}. In this case the contribution to $\chi_y$ genus is 
\begin{equation*}
\begin{split}
d^{2g} & y^{-m-\frac{n-g}{2}} \cdot \chi_{-y} \left(  \frac{[\textup{Sym}^m(X\times \mathbb{P}^{d-1})]}{[A]} \right) \\
 = & n^{g-1} d^{g+1}y^{-m-\frac{n-g}{2}}\cdot \chi_{-y^m}(Y) \left( \sum_{i=0}^{d-1} y^{im} \right) \left( \sum_{j=0}^{m-1} y^j    \right)^g
 \end{split}
\end{equation*}
by Lemma \ref{Kummerchiy}.
\end{proof}

\begin{remark}
Corollary \ref{euler2} together with Lemma \ref{euler1} gives us a way to compute the Euler characteristic $\chi(K_n(X))$ in dimension 3. The results here agree with Corollary \ref{Gulconj}.
\end{remark}

\subsection{Properties of virtual Hodge numbers}\label{propsofhodge}
We study Hodge numbers of the virtual motive of the generalized Kummer scheme. Define the normalized virtual motive to be
\begin{equation*}
[\widetilde{K}_n(X)] : = \cL^{\frac{3n-g}{2}} \cdot [K_n(X)]_{\textup{vir}}
\end{equation*}
for an abelian geometry $X \rightarrow A$. If $A$ is chosen to be $\textup{Spec}(\cC)$ we write $[\widetilde{X}_n] := [K_n(X)]_{\textup{vir}} = [\textup{Hilb}^n(X)]_{\textup{vir}}$. We obtain from  Theorem \ref{thm3fold} that
\begin{equation} \label{normalizedKummer}
[A \times \widetilde{K}_n(X)] = \sum_{\alpha \vdash n} g(\alpha)^{2g} \cdot \cL^{n-l(\alpha)} \prod_{m\geq 1} [\textup{Sym}^{b_m(\alpha)}{(X\times \mathbb{P}^{m-1})}].
\end{equation}
holds in $K_0(\textup{HS})$.
The following statements show that the Hodge numbers of virtual motives behave like those of $(3n-g)$-dimensional smooth projective varieties.
\begin{enumerate}
\item[(a).](Hodge diamond) The nontrivial Hodge numbers $h^{p,q}$ satisfy $p,q \in \mathbb{Z}$ and $0 \leq p,q\leq 3n-g$.
\item[(b).] (Symmetry) We have $h^{p,q} = h^{q,p}$. If $X$ is smooth projective, we also have
\[
h^{p,q} = h^{3n-g-p, 3n-g-q}.
\]
\item[(c).] (Lefschetz) When $X$ is smooth projective, the inequality $h^{p-1,q-1} \leq h^{p, q}$ holds for $p+q \leq{3n-g}$. 
\end{enumerate}

\begin{proof}
It is clear that (a) and $h^{p,q} = h^{q,p}$ follow directly from the expression (\ref{normalizedKummer}). Now assume $X$ is smooth projective. Consider the addition map \[\sigma_{\alpha}: \prod_{m \geq 1} \left( X \times \mathbb{P}^{m-1} \right)^{b_m(\alpha)} \rightarrow A\] defined by
\[
\left( (x_j^1)_j, (x_j^2)_j , \dots   \right)  \mapsto \sum_{i,j}i\cdot x_j^i,
\]
where the sum is taken after projecting to the abelian variety $A$. The map $\sigma_\alpha$ induces a map
\[
\tilde{\sigma}_\alpha: \prod_{m\geq 1}\textup{Sym}^{b_m(\alpha)}\left(  X \times \mathbb{P}^{m-1}    \right)\rightarrow A.
\]
Note that the fiber $\sigma_\alpha^{-1}(0)$ is smooth projective, and $\tilde{\sigma}_\alpha^{-1}(0)$ is a finite group quotient of $\sigma^{-1}(0)$. Hence $\tilde{\sigma}^{-1}(0)$ satisfies weak and hard Lefschetz. In particular, we have
\[
h^{p,q}([\tilde{\sigma}_\alpha^{-1}(0)]) = h^{d-p ,d-q}([\tilde{\sigma}_\alpha^{-1}(0)])
\]
where $d = n+2l(\alpha)-g$ and 
\[
h^{p-1,q-1}([\tilde{\sigma}_\alpha^{-1}(0)]) \leq h^{p,q}([\tilde{\sigma}_\alpha^{-1}(0)]) \]
 when $p+q<\frac{n-g}{2}+l(\alpha)$. The statements are then proven by the following equation (see (\ref{fiber3fold})):
\[
[\widetilde{K}_n(X)] = \sum_{\alpha \vdash n} \cL^{n-l(\alpha)} \cdot [\tilde{\sigma}_\alpha^{-1}(0)]. \qedhere
\]
\end{proof}

Finally we compute stable Hodge numbers. The method follows \cite{Lecture} Corollaries 2.3.13 and 2.4.14.

For convenience, we assume $X$ is connected. In this case $h^{0,0}([X]) = 0$ or 1. By \cite{BBS} and the specialization of power structures \cite{GZLMH2} we have the following equation
\begin{equation} \label{hodgenumber}
\sum_{n \geq 0} E (   [\widetilde{X}_n]; x,y ) \cdot t^n = \prod_{m\geq 1} \prod_{k=0}^{m-1}\prod_{p,q} \left(  \frac{1}{1-x^{k+m+p-1}y^{k+m+q-1}t^m}                   \right)^{e^{p,q}([X])}.
\end{equation}

First we study the case $h^{0,0}([X]) =1$ (e.g. $X$ is projective). The righthand side of (\ref{hodgenumber}) can be written as $(1-t)^{-1}\cdot G(x,y,t)$ where $G(x,y,t)$ is a rational function in the variables $x,y,t$ and $(1-t)$ does not divide the denominator of $G(x,y,t)$. The following proposition determines the stable Hodge numbers of $[\widetilde{K}_n(X)]$ for the abelian geometry $X \to A$.

\begin{proposition}\label{stable}
\begin{enumerate}
\item
When $n> 2p$ or $n>2q$, we have
\[
h^{p,q}(   [\widetilde{K}_n(X)] ) = (-1)^{p+q}\cdot\textup{Coeff}_{x^py^q}\left( \frac{E(   [\widetilde{X}_n]; x,y )}{[(1-x)(1-y)]^g}   \right).
\]
\item When $h^{0,0}([X])=1$, and $n \geq 2p$ or $n \geq 2q$, we have
\[
h^{p,q}( [\widetilde{X}_n]) = (-1)^{p+q}\cdot \textup{Coeff}_{x^py^q}\left(   G(x,y,1)         \right).
\]
\end{enumerate}
\end{proposition}
\begin{proof}
Since $g(\alpha)= 1$ if $l(\alpha) > \frac{n}{2}$, we get from the expression (\ref{normalizedKummer}) that
\[
\textup{Coeff}_{x^py^q}\left(  E([\widetilde{K}_n(X) \times A];x,y) \right)
= \textup{Coeff}_{x^py^q} \left(  E([\widetilde{X}_n]; x,y)        \right).
\]
This proves (1). We prove (2) by explicit computation. Note that if $n\geq 2p$ or $n\geq 2q$, from (\ref{hodgenumber}) we know $\textup{Coeff}_{x^py^qt^j}\left( G(x,y,t) \right) =0$ for $j>n$. Hence
\begin{equation*}
\begin{split}
e^{p,q}([\widetilde{X}_n]) = & \textup{Coeff}_{x^py^qt^n}\left(  \left(  \sum_{i=0}^{\infty}t^i  \right) G(x,y,t) \right)\\
= & \sum_{j=0}^{n} \textup{Coeff}_{x^py^qt^j}\left(  G(x,y, t) \right)\\
= &\sum_{j=0}^{\infty} \textup{Coeff}_{x^py^qt^j}\left( G(x,y,t)   \right)\\
= & \textup{Coeff}_{x^py^q}\left(  G(x,y,1)  \right). \qedhere
\end{split}
\end{equation*}
\end{proof}

If $h^{0,0}([X]) = 0$, Proposition \ref{stable} (1) still holds. The stable hodge numbers $h^{p,q}([\widetilde{X}_n])$ are trivial by the following proposition.
\begin{proposition}
When $h^{0,0}([X]) = 0$, and $n \geq 2p$ or $n \geq 2q$, we have
\[
h^{p,q}([\widetilde{X}_n]) = 0.
\]
\end{proposition}
\begin{proof}
Since $e^{0,0}([X]) = 0$, the coefficient of $x^py^qt^n$ on the righthand side of (\ref{hodgenumber}) vanishes if $n \geq 2p$ or $n \geq 2q$.
\end{proof}

\section{Torsion Sheaves}
Let $\mathcal{T}_{n}(X)\to A$ be the moduli stack of length $n$ torsion sheaves on $X$ with the map to $A$ given by projecting to $A$ and taking the sum of the support of the sheaf with multiplicities. Collecting the motives of these stacks into a series $\mathcal{T}_X(t) = \sum_{n\geq 0} [\mathcal{T}_{n}(X)\to A]t^n$ it follows from Lemma \ref{geomstack} that we have
\[ \mathcal{T}_X(t) = \left( \sum_{n\geq 0} [\mathcal{H}_{n,r+g}\to 0\in A]t^n \right)^{[X\to A]} \]
where $\mathcal{H}_{n,r+g}$ is the stack of length $n$ torsion sheaves supported at the origin. We hereby denote as $\mathcal{H}_{r+g}(t)$ the sum inside the above bracket. 

Expanding $\mathcal{H}_{r+g}(t)$ as a formal series as in Section \ref{gensection} we see that all the Hodge numbers of $\mathcal{T}_{n}(X)$ are expressible in terms of the Hodge numbers of symmetric products of $X$. So again these numbers are deformation invariants of $X$.

\subsection{Curves} Here we assume $r+g =1$. The sum inside the bracket is given by Euler's formula \cite[Section 3.3]{BM} and equals
\[ \mathcal{H}_{1}(t) = \prod_{k = 1}^{\infty} (1-\mathbb{L}^{-k}t)^{-1}. \]

\begin{proposition}\label{stackcurve} Let $X\to A$ be a curve then
\[  [\mathcal{T}_{n}(X)\to A] = \sum_{\substack{ l, 0< r_1,\ldots , r_{l}  \\  n = r_1 + \cdots + r_l }} \prod_{i=1}^l\frac{ [\Sym^{r_i}(X)\to A] }{ \mathbb{L}^{s_i(r_1, r_2,\ldots)} -1 } \] 
where $s_i(r_1,r_2,\ldots) = \sum_{j\geq i}r_j$.
\end{proposition}
\begin{proof} Expanding the generating series $ \mathcal{T}_X(t)$ we get a sum over terms
\[ \prod_{i=1}^l [\textup{Sym}^{r_i}(X)\to A] \mathbb{L}^{-r_ik_i} t^{r_i} = \prod_{i=1}^l [\Sym^{r_i}(X)\to A] \mathbb{L}^{-s_i(r_1,\ldots, r_l)(k_i-k_{i-1})} t^{r_i} \]
indexed by $r_i,k_i\in \mathbb{N}$ such that $k_{i+1}-k_i> 0$. Now compare the $t^n$ coefficients.
\end{proof}
Note that given the curve is connected we have the following genus specific description of the motives $[\Sym^n(X)\to A]$ equal to
\begin{description}
  \item[$g(X)<1$] $[ \mathbb{P}_k^n \to 0 ]$,
  \item[$g(X)=1$] $[ A \to A ] \cdot [\mathbb{P}^{n-1}_k\to 0\in A ]$,
  \item[$g(X)>1$] $[ \textup{Jac}(X)\to 0 ] \cdot [\mathbb{P}^{n-g}_k\to 0]$ (when $n>2g -1$).
\end{description}

\subsection{Surfaces} Here we assume $r+g =2$. The sum inside the bracket can be deduced from a geometric consideration of commuting varieties \cite[Section 3.1]{BM}
\[ \mathcal{H}_{2}(t) = \prod_{m = 1}^{\infty} \prod_{k = 1}^{\infty} (1-\mathbb{L}^{-k}t^m)^{-1}. \]
\begin{proposition}\label{stacksurface} Let $X\to A$ be a surface then
\[  [\mathcal{T}_{n}(X)\to A] = \sum_{  \substack{ 0 < r_1^m,\ldots , r_{l_m}^m \\ n = \sum_{i,m} mr_i^m  }  } \prod_{m\geq1}m_*\left( \prod_{i=1}^{l_m}  \frac{ [\Sym^{r_i^m}(X)\to A] }{ \mathbb{L}^{s_i(r_1^m, r_2^m,\ldots)} -1 } \right)\]
where $s_i(r_1,r_2,\ldots) = \sum_{j\geq i}r_j$.  
\end{proposition}
\begin{proof} Expand series as in Proposition \ref{stackcurve}. \end{proof}

\subsection{Threefolds} Here we assume $r+g =3$. In this case their are no known formulas for the motivic class of the stack of sheaves supported at a point. Instead consider the virtual motives $\mathcal{T}_X^{\textup{vir}}(t) = \sum_{n\geq 0} [\mathcal{T}_{n}(X)\to A]_{\textup{vir}}t^n$ by definition \cite{BBS} we have $\mathcal{T}_X^{\textup{vir}}(t)= \left( \mathcal{H}_{3}^{\textup{vir}}(t) \right)^{[X\to A]}$ where 
\[ \mathcal{H}_{3}^{\textup{vir}}(t) = \left( \mathcal{H}_{2}(t) \right)^{\frac{[\C^2]}{[\C^3]}} = \prod_{m= 1}^\infty\prod_{k=1 }^\infty  (1-\mathbb{L}^{-k-1}t^m)^{-1} \]
is the series for the virtual motives of the stack of sheaves supported at a point.
\begin{proposition}\label{stack3fold} Let $X\to A$ be a threefold then
\[  [\mathcal{T}_{n}(X)\to A]_{\textup{vir}} = \sum_{  \substack{ 0 < r_1^m,\ldots , r_{l_m}^m \\ n = \sum_{i,m} mr_i^m  }  }  \prod_{m\geq1}m_*\left( \prod_{i=1}^{l_m}  \mathbb{L}^{-r_{i}^m}\frac{ [\Sym^{r_i^m}(X)\to A] }{ \mathbb{L}^{s_i(r_1^m, r_2^m,\ldots)} -1 } \right)\]
where $s_i(r_1,r_2,\ldots) = \sum_{j\geq i}r_j$.  
\end{proposition}
\begin{proof} A slight modification of the surface case. \end{proof}

\subsection{Kummer fibers} To extract the Kummer fibers from Proposition \ref{stackcurve}, Proposition \ref{stacksurface}, or Proposition \ref{stack3fold} one can use this simple lemma extending Lemma \ref{symkummer}.
 
\begin{lemma}\label{sym2kummer} Let $X\to A$ be a Zariski trivial fiber bundle over a $g$-dimensional complex abelian variety $A$ then,
\[ [A]\cdot0^*\left( \prod_{m\geq 1} m_*\left(\prod_{i=1}^{l_m}[\textup{Sym}^{r^m_i}(X)\to A] \right)\right) = g(r)^{2g} \cdot  \prod_{m,i} [\textup{Sym}^{r^m_i}(X)]  \]
in the ring $K_0(\textup{HS})$ of Hodge structures with $g(r) = \textup{gcd}(\{ m: r^m_i \neq 0 \textrm{ for some } i\}).$
 \end{lemma}
 \begin{proof}
 Here the proof follows that of Lemma \ref{symkummer}. The only difference is that we work with the invariant cohomology for the Young subgroups $S_{r_1^m}\times\cdots \times S_{r_{l_m}^m}\subset S_{\sum_i r_i^m}$ instead of the full symmetric groups.
 \end{proof}
 Let us denote by $[\mathcal{K}_n(X)]$ the class of the Kummer fiber.  Assuming $g\geq 1$ we record the following formulas for the $\chi_y$ genera of these stacks.

\begin{corollary} 
For $X\to A$ a curve, surface, or 3-fold with fiber $Y$ and $g\geq 1$, we have in order
\begin{eqnarray*}
\chi_{-y}([\mathcal{K}_n(X)]) &=& \frac{1}{y-1} \\
\chi_{-y}([\mathcal{K}_n(X)]) &=&  \frac{n^{g-1}}{y-1} \sum_{d\cdot m = n} \chi_{-y^m}(Y) \cdot d^{g+1} \left( \sum_{i=0}^{m-1} y^i \right)^{g-1} \\
\chi_{-y}([\mathcal{K}_n(X)]_{\textup{vir}}) &=&   \frac{n^{g-1}}{y-1} \sum_{d\cdot m = n} y^{-m} \chi_{-y^m}(Y) \cdot d^{g+1} \left( \sum_{i=0}^{m-1} y^i \right)^{g-1}.
\end{eqnarray*}
\end{corollary}
\begin{proof}Follows from Proposition \ref{stackcurve}, Proposition \ref{stacksurface}, and Proposition \ref{stack3fold} together with Lemma \ref{Kummerchiy} and Lemma \ref{sym2kummer}. 
\end{proof}

\end{document}